\documentclass[12pt]{amsart}
\usepackage{amsthm,amsmath,amssymb}
\input xypic

\newcommand{\PP}{{\mathbb{P}}}
\newcommand{\QQ}{{\mathbb{Q}}}

\newcommand{\ZZ}{{\mathbb{Z}}}

\newcommand{\calA}{{\mathcal A}}

\newcommand{\calM}{{\mathcal M}}
\newcommand{\calJ}{{\mathcal J}}

\newcommand{\calX}{{\mathcal X}}

\newcommand{\op}{\operatorname}

\newcommand{\Pic}{\op{Pic}}
\newcommand{\PPic}{\op{{\mathcal P}ic}}

\newcommand{\Mgn}{{\calM_{g,n}}}
\newcommand{\Mgnct}{{\calM_{g,n}^{ct}}}
\newcommand{\oMgn}{{\overline{\calM}_{g,n}}}

\newcommand{\Jg}{{\mathcal J}_g}
\newcommand{\ud}{{\underline{d}}}
\newcommand{\Ell}{{\mathcal E}}

\theoremstyle{plain}
\newtheorem{thm}{Theorem}

\newtheorem{prop}[thm]{Proposition}
\newtheorem{cor}[thm]{Corollary}

\theoremstyle{definition}

\begin{document}
\title{The double ramification cycle and the theta divisor}
\author{Samuel Grushevsky}
\address{Mathematics Department, Stony Brook University, Stony Brook, NY 11794-3651, USA.}
\email{sam@math.sunysb.edu}
\thanks{Research of the first author is supported in part by National Science Foundation under the grant DMS-10-53313.}
\author{Dmitry Zakharov}
\address{Mathematics Department, Stony Brook University, Stony Brook, NY 11794-3651, USA.}
\email{dvzakharov@gmail.com}

\begin{abstract}

We compute the classes of universal theta divisors of degrees zero and $g-1$ over the Deligne-Mumford compactification $\oMgn$ of the moduli space of curves, with various integer weights on the points, in particular reproving a recent result of M\"uller \cite{muller}.

We also obtain a formula for the class in $CH^{g}(\Mgnct)$ (moduli of stable curves of compact type) of the double ramification cycle, given by the condition that a fixed linear combination of the marked points is a principal divisor, reproving a recent result of Hain \cite{hainnormal}.

Our approach for computing the theta divisor is more direct, via test curves and the geometry of the theta divisor, and works easily over all of $\oMgn$. We used our extended result in \cite{futuregz} to study the partial compactification of the double ramification cycle.

\end{abstract}

\maketitle

\section{Introduction}

Let $\Mgn$ denote the moduli space of smooth genus $g$ curves with $n$ labeled distinct marked points, let $\Mgnct$ denote its partial compactification by stable curves of compact type, and let $\oMgn$ denote the Deligne-Mumford compactification by stable curves. Let $\Jg^d\rightarrow \calM_g^{ct}$ denote the universal family of Picard varieties (Jacobians) of degree $d$ (recall that the Jacobian of a stable curve of compact type is in fact an abelian variety), and by abuse of notation let $\Jg^d\rightarrow \Mgnct$ also denote its pullback under the forgetful map $\pi:\Mgnct\rightarrow\calM_g^{ct}$. For any collection of integers $\ud=(d_1,\ldots,d_n)\in\ZZ^n$ of total degree $\deg\ud:=d_1+\ldots+d_n=d$, define the section $s_\ud:\calM_{g,n}\to\Jg^d$ by $s_\ud(C,p_1,\ldots,p_n)=\sum d_i p_i\in \Pic^d(C)$.

The degree $g-1$ universal Picard variety has the universal theta divisor $\Theta\subset\Jg^{g-1}$, while the degree $0$ universal Picard variety has the zero section $z_g:\calM_g^{ct}\rightarrow \Jg^0$. We denote by $T$ the universal symmetric theta divisor $T\subset \calJ_g^0$ trivialized along the zero section. In this paper, we compute the pullbacks of $\Theta$ and $T$ to $\oMgn$.

Our main result is given in Theorem~\ref{thm:main}.

Our motivation for computing the pullbacks of the theta divisors is that they can be used to compute the classes of natural geometric loci on the moduli space of curves, which have been studied recently. One example is a result proved (in cohomology) by Hain in \cite{hainnormal} (see Section~2 for notation):
\begin{thm}\label{thm:0}
For $\deg\ud=0$, the class in $CH^{g}(\Mgnct,\QQ)$ of the pullback of the zero section of the universal Jacobian variety $\calJ_g^0$ is equal to
$$
[s_\ud^*z_g]=\frac{1}{g!}\left[\frac{1}{2}\displaystyle\sum_{i=1}^n d_i^2 K_i-\frac{1}{2}\displaystyle\sum_{P\subseteq I}
\left(d_P^2-\displaystyle\sum_{i\in P}d_i^2\right)\delta_0^P-\frac{1}{2}
\displaystyle\sum_{h> 0,P\subseteq I} d_P^2
\delta_h^P\right]^g,
$$
where $d_P=\sum_{i\in P} d_i$.
\end{thm}
The cycle $s^*_\ud z_g$ is known as the {\it double ramification cycle}, and the question of computing its class is due to Eliashberg. Geometrically, it can be interpreted as the locus of curves admitting a map to $\PP^1$ with ramification multiplicities being all the positive $d_i$ over $0$, and all the negative $d_i$ over $\infty$.

The relationship of the double ramification cycle to the pullback of the theta divisor is as follows. Let $p:\calX_g\rightarrow\calA_g$ denote the universal family of principally polarized abelian varieties, let $z_g$ denote the zero section of this family, and let $T\subset\calX_g$ denote the universal symmetric theta divisor trivialized along the zero section. Then
\begin{equation}
[z_g]=\frac{[T]^g}{g!}\in CH^{g}(\calX_g,\QQ)\label{equ:0section}.
\end{equation}
Note that the pushforwards $p_*([T]^k)$ were considered by Mumford and studied in detail by van der Geer \cite{vdgeercycles}, but (\ref{equ:0section}), in the Chow ring, follows from the existence of a multiplicative decomposition for $Rp_*\QQ$, proved by Deninger and Murre \cite{denmur} (see eg.~\cite[Cor.~16.5.7]{bila}, or \cite[Prop.~4.3.6, Cor.~4.3.9]{voisinnotes}). In cohomology (\ref{equ:0section}) was proven independently by
Hain {\cite[Prop.~8.1]{hainnormal}}.

\smallskip
A closely related result is the following very recent theorem of M\"uller \cite{muller}, which is nearly equivalent to computing the pullback of the degree $g-1$ theta divisor (our notation is slightly different, see Section~2 for details):

\begin{thm}\label{thm:muller}
For any $\ud$ with $\deg \ud=g-1$, and such that at least one $d_i$ is negative, define the following locus in $\calM_{g,n}$:
$$
 D_{\ud}=\left\{(C,p_1,\ldots,p_n)\in\calM_{g,n}\mid h^0(C,d_1p_1+\cdots+d_np_n)\geq 1\right\},
$$
and let $\overline{D}_{\ud}$ denote its closure in $\oMgn$. Then in $\Pic_\QQ(\oMgn)$ the class of $\overline{D}_{\ud}$ is equal to
$$
  [\overline{D}_{\ud}]=-\lambda_1+0\cdot\delta_{irr}+\frac{1}{2}\displaystyle\sum_{i=1}^n d_i(d_i+1) K_i-\frac{1}{2}\displaystyle\sum_{P\subseteq I}\left(d_P^2-\displaystyle\sum_{i\in P}d_i^2\right)\delta_0^P-\\
$$
$$
  -\frac{1}{2}\displaystyle\sum_{h> 0,P\subseteq I} (d_P-h)(d_P-h+1)\delta_h^P-
\displaystyle\sum_{P\subseteq P_+,h>d_P}(h-d_P)\delta_h^P,
$$
where $P_+:=\{i\in I \mid d_i\geq 0\}$.
\end{thm}

\medskip
Theorem~\ref{thm:0} was first proved by Hain in April 2003, using normal functions. His proof became more widely known (and we became aware of it) in February 2011 with the appearance of Hain's preprint \cite{hainnormal}. Over the locus of curves with rational tails $\calM_{g,n}^{rt}$, this result was reproved by Cavalieri, Marcus, and Wise \cite{cavalieri} in July 2011, using Gromov--Witten theory. We obtained a proof of Theorem~\ref{thm:0} in June 2011, and discussed it with Hain and others at PCMI in July 2011, in particular correcting Hain's original formula, see \cite{hainnormal}. Theorem~\ref{thm:muller} was proved by M\"uller \cite{muller} in March 2012, and we decided to make our results available due to a continued interest in the problem. 

Our method is more elementary, and could also be applied in various similar situations computing classes related to the theta divisor. In particular in our paper \cite{futuregz} we consider degenerations of abelian varieties and prove an extension of Theorem~\ref{thm:0}, and primarily of formula (\ref{equ:0section}), to the universal family of semiabelic varieties of torus rank one, using our computation of the class $[s_\ud^*T]$ on $\oMgn$.
\medskip

We calculate the class of the theta divisor using test curves. In Section~2 we describe a basis of test curves and a basis for $\Pic_\QQ(\oMgn)$, and compute their intersection numbers. This is a standard method, and the calculations are standard, but we briefly summarize them for convenience and for future reference --- the result is given in Proposition~\ref{prop:standardintersect}. The more interesting part is computing the intersections of the test curves with the pullback of the theta divisors $\Theta$ and $T$ under the map $s_\ud$, which is done by using the properties of the theta function and the Abel--Jacobi map. This is the content of Proposition~\ref{prop:thetaintersect}. Finally, in Theorem~\ref{thm:main} we derive the formulas for the pullbacks of the theta divisors.

We follow the standard notation and conventions for working on the moduli of curves, referring for example to \cite{mumfordtowards},\cite{hamobook},\cite{acgh2} for known results, discussion, and further references.

\section{Divisors, test curves, and intersection numbers}

In this section we describe a basis of divisor classes on $\oMgn$ and a collection of test curves, and we compute their intersection numbers. Our computation technique is quite standard, but we include it for the sake of completeness and for possible reference value.

Let $\pi:\oMgn\to\overline{\calM}_g$ be the forgetful map and let $\pi_i:\oMgn\to\overline{\calM}_{g,1}$ be the map that forgets all but $i$-th marked point. Let $I=\{1,2,\ldots,n\}$ denote the indexing set. For a subset $P\subseteq I$ let $P^c$ denote its complement. We consider the following set of generators for $\Pic_\QQ(\oMgn)$:
\begin{itemize}
\item The classes $K_i=c_1(\pi_i^*(\omega_p))$, where $\omega_p$ is the relative dualizing sheaf of the universal curve $p:\overline{\calM}_{g,1}\rightarrow\overline{\calM}_g$.
\item The classes $\delta_h^P$ of the boundary divisors $\Delta_h^P$, where $P\subseteq I$ and $0\leq h\leq g$. The generic point of $\Delta_h^P$ is a reducible curve consisting of a smooth component of genus $h$ containing the marked points indexed by $P$ and a smooth component of genus $g-h$ with the remaining points, joined at a node. To satisfy the stability condition we assume that $|P|\geq 2$ if $h=0$ and $|P|\leq g-2$ if $h=g$. Note that $\delta_h^P=\delta_{g-h}^{P^c}$. In any sum involving $\delta_h^P$, we assume that each divisor class appears only once, so we either explicitly state which one we are adding or we sum with symmetric coefficients. This convention is used by M\"uller \cite{muller} but not by Hain \cite{hainnormal}. We use $\delta$ to denote the divisor classes on the moduli stack and $\Delta$ to denote the divisors on the coarse moduli space.
\item The class $\delta_{irr}$ of the divisor $\Delta_{irr}$. The generic point of $\Delta_{irr}$ is a smooth curve of genus $g-1$ with two points identified to form a node.
\item The first Chern class $\lambda_1$ of the Hodge bundle.
\end{itemize}

The above classes are known to be a basis of $\Pic_\QQ(\oMgn)$ for $g\ge 3$. In what follows we assume that $g\ge 3$, but our results also hold for $g=1$ and $g=2$ by inspection.

A more common choice of a basis replaces the classes $K_i$ with the classes $\psi_i$ of the cotangent bundles at the marked points $p_i$. We use the classes $K_i$ in our final result, but we use both $K_i$ and $\psi_i$ in intermediate calculations. These two classes differ by a linear combination of boundary classes (see \cite[p.~161]{ardcpic}):
\begin{equation}\label{equ:Kpsi}
\psi_i=K_i+\displaystyle\sum_{P\ni i, |P|\geq 2}\delta_0^P.
\end{equation}
Note that the classes $K_i$ are denoted $\psi_i$ in \cite{hainnormal}, while in \cite{muller} the $\psi_i$ have the same meaning as above.

We now define a collection of test curves on $\oMgn$.
\begin{itemize}
\item {\bf The curves $Z_i$.} Let $1\leq i\leq n$, and fix a generic smooth curve $(C,p_1,\ldots,\widehat{p}_i,\ldots,p_n)$ in $\mathcal{M}_{g,n-1}$. We define the family $Z_i\subset\oMgn$ by letting the point $p_i$ range over $C$. The curve $Z_i$ is isomorphic to $C$.
\item {\bf The curves $Z_h^P$.}  Let $P=\{i_1,\ldots,i_m\}\subseteq I$ be a subset, let $P^c=\{j_1,\ldots,j_{n-m}\}$ be the complement, and let $0\leq h<g$, where we assume that $m\geq 2$ if $h=0$ and $m<n$ if $h=g-1$. Fix a generic smooth curve $(C_1,p_{j_1},\ldots,p_{j_{n-m}})$ in $\calM_{g-h,n-m}$, and define the family $Y_h^P\subset \overline{\calM}_{g-h,n-m+1}$  by adding a point $q_1$ and letting it range over $C_1$. Now fix another generic smooth curve $(C_2,p_{i_1},\ldots, p_{i_m},q_2)\in\calM_{h,m+1}$, and define the family $Z_h^P\subset \oMgn$ by attaching $q_1$ to $q_2$ to form a node. The curves $Y_h^P$ and $Z_h^P$ are both isomorphic to $C_1$.
\item {\bf The curve $\Ell$.} Fix a generic smooth curve $(C_2,p_1,\ldots,p_n,q_2)$ in $\calM_{g-1,n+1}$, and let $\Ell$ be the family obtained by attaching a varying elliptic curve $(C_1,q_1)\in\overline{\calM}_{1,1}$ to the curve $C_2$ at $q_2$. We consider $\Ell$ to be ``stacky'', i.e.~since $(C_1,q_1)$ has an involution, we consider the generic point of $\Ell$ with coefficient $1/2$.
\item {\bf The curve $Z_{irr}$.} Fix a generic smooth curve $(C_1,p_1,\ldots,p_n,q_1)$ in $\calM_{g-2,n+1}$, and consider, for a fixed generic elliptic curve $(E,q_2)$ in $\calM_{1,1}$, the family $Z_{irr}$ obtained by varying a point $q_3$ over $E$, and attaching $q_1,q_2,q_3$ to a rational curve. The curve $Z_{irr}$ is isomorphic to $E$.
\end{itemize}

We now compute, in the standard way, the intersection numbers of these test curves with the chosen basis of divisors, which will then imply that these curves form a basis for $N_1(\oMgn)$.
\begin{prop} \label{prop:standardintersect}
The test curves have the following intersection numbers with the divisors, where we write $(P,h)=(Q,l)$ if $P=Q$ and $h=l$ or if $P^c=Q$ and $g-h=l$.
$$
Z_i\cdot K_j=
\left\{\begin{array}{cc}\!\! 2g-2, & i=j,\\ 0, &\!\!\! \mbox{otherwise},\end{array}\right.
\ \ Z_i\cdot\delta_h^P=\left\{\begin{array}{cc} 1, & \!\!(P,h)=(\{i,j\},0),j\neq i\\
0, & \mbox{otherwise,}\end{array}\right.
$$
$$
Z_i\cdot\delta_{irr}=0, \quad Z_i\cdot\lambda_1=0,\quad
Z_h^P\cdot K_i=\left\{\begin{array}{cc} 2g-2, & h=0\mbox{ and }i\in P,\\ 1, & h>0
\mbox{ and }i \notin P, \\
0, & \mbox{otherwise},\end{array}\right.
$$
$$
Z_h^P\cdot\delta_l^Q =\left\{\begin{array}{cc} 2-2(g-h)-|P^c|, & (Q,l)=(P,h),\\
1, & (Q,l)=(P\sqcup \{j\},h),j\in P^c,\\0, & \mbox{otherwise,}\end{array}\right.
$$
$$
Z_h^P\cdot\delta_{irr}=0,\ \ Z_h^P\cdot\lambda_1=0,\ \ \Ell\cdot K_i=0, 
\ \ \Ell\cdot\delta_h^P=\left\{\begin{array}{cc}\!\! -1/24, &\!\! (P,h)=(\emptyset,1),\\
0, & \mbox{otherwise,}\end{array}\right.
$$
$$
\Ell\cdot\delta_{irr}=1/2,\quad
\Ell\cdot\lambda_1=1/24,\quad Z_{irr}\cdot \delta_h^P=\left\{\begin{array}{cc} 1, & (P,h)=(\emptyset,1),\\
0, & \mbox{otherwise,}\end{array}\right.
$$
$$
Z_{irr}\cdot K_i=0,\quad Z_{irr}\cdot \delta_{irr}=-1,\quad Z_{irr}\cdot \lambda_1=0.
$$
\end{prop}
\begin{proof} We first compute the intersections with the test curves $Z_i$. If $i\neq j$, then $\pi_j(Z_i)=\{(C,p_j)\}$ is a single point in $\calM_{g,1}$, so $K_j$ restricts to a trivial line bundle on $Z_i$ and $Z_i\cdot K_j=0$. On the other hand, $\pi_i(Z_i)$ is the fiber $\{(C,p_i)\mid p_i\in C\}$ of $\calM_{g,1}$ over $C\in \calM_g$, so the bundle $K_i$ restricts to the cotangent bundle of $C$, which has degree $2g-2$. The curves parametrized by $Z_i$ are smooth except when $p_i=p_j$ for some $j\neq i$, in which case the two marked points lie on a rational tail. This gives the intersection numbers with the boundary divisors. Finally, $\pi(Z_i)$ is a point, so the Hodge bundle is trivial on $Z_i$ and $Z_i\cdot \lambda_1=0$.

The remaining test curves are all supported on the boundary, so to compute the intersection numbers we use the technique of \cite{faberalgorithms}. The boundary divisor $\Delta_h^P$ is the image of the product $\overline{\calM}_{g-h,P^c\sqcup\{r_1\}}\times \overline{\calM}_{h,P\sqcup\{r_2\}}$ under the map identifying $r_1$ and $r_2$. We compute the intersection numbers with the test curves by pulling back to this product. We denote by $pr_1$ and $pr_2$ the projection maps to the two components. The class $\psi_i$ on $\overline{\calM}_{g,n}$ pulls back to either $pr_1^*\psi_i$ if $i\in P^c$ or $pr_2^*\psi_i$ if $i\in P$. According to \cite{faberalgorithms}, the pullback of the divisor class $\delta_h^P$ to the product $\overline{\calM}_{g-h,P^c\sqcup\{q_1\}}\times \overline{\calM}_{h,P\sqcup\{q_2\}}$ is either $-pr_1^*\psi_{q_1}-pr_2^*\psi_{q_2}+pr_2^*\delta_{g-2h}^{I\sqcup\{q_2\}}$ if $g-2h\geq 0 $ and $P=I$ or $-pr_1^*\psi_{q_1}-pr_2^*\psi_{q_2}$ otherwise.

To compute $Z_h^P\cdot K_i$ for $h>0$, we note that the family $Z_h^P$ does not parametrize any curves with rational tails, so by (\ref{equ:Kpsi}) we have $Z_h^P\cdot K_i=Z_h^P\cdot\psi_i$. The curve $Z_h^P$ pulls back to $Y_h^P\times pt$ on the product, therefore $Z_h^P\cdot\psi_i=0$ if $i\in P$. If $i\in P^c$, then passing to the first factor in the product we have
$$
Z_h^P\cdot\psi_i=Y_h^P\cdot \psi_i=Y_h^P\cdot \left(K_i+\displaystyle\sum_{i\in Q\subseteq P^c\sqcup\{r_1\}}\delta_0^Q\right)=1,
$$
since $Y_h^P\cdot K_i=0$ as above, and the only boundary divisor in the sum that intersects $Y_h^P$ is $\delta_0^{\{p_i,r_1\}}$.

For the intersection numbers $Z_0^P\cdot K_i$, we note that the projection $\pi_i$ collapses rational tails, so the image $\pi_i(Z_0^P)$ is the point $(C_2,p_i)\in\overline{\calM}_{g,1}$ if $i\in P^c$, or the curve $\{(C_2,q_1)\mid q_1\in C_2\}$ if $i\in P$. Therefore, $Z_0^P\cdot K_i$ is zero in the first case and $2g-2$ in the second.

The intersections of the test curves $Z_h^P$ with the boundary divisors correspond to the possible degenerations of the parameterized curves. All of the non-empty intersections are transverse and equal to one, except that the curve $Z_h^P$ lies on the divisor $\Delta_h^P$. Therefore, we again restrict to the first factor and obtain
$$
Z_h^P\cdot\delta_h^P=-Y_h^P\cdot\psi_{r_1}=-Y_h^P\cdot\left(K_{r_1}+\displaystyle\sum_{Q\subseteq P^c, Q\neq\emptyset}\delta_0^{Q\sqcup\{r_1\}}\right)=
$$
$$
=-Y_h^P\cdot K_{r_1}-\displaystyle\sum_{k\in P^c}Y_h^P\cdot\delta_0^{\{p_k,r_1\}}=-(2(g-h)-2)-|P^c|.
$$
Finally, the Hodge bundle on $Z_h^P$ is trivial, so $Z_h^P\cdot \lambda_1=0$.

The curve $\Ell$ lies in the divisor $\Delta_1^{\emptyset}$ and intersects $\Delta_{irr}$. The corresponding intersection numbers, as well as $\Ell\cdot\lambda_1$, were computed by Wolpert in \cite{wolperthomology} (note, however, that for us $\delta_1^\emptyset$ and $\Ell$ are stacky, so the intersection numbers differ, see also \cite[Lemma 4.2]{muller}). Finally, for any $i$ the curve $\pi_i(\Ell)$ is a fixed curve with a fixed marked point and an attached varying elliptic tail, so $\Ell\cdot K_i=0$.

The curve $Z_{irr}$ lies in the boundary divisors $\Delta_2^{\emptyset}$ and $\Delta_{irr}$, intersects $\Delta_1^{\emptyset}$ at one point, and does not intersect the other boundary divisors.
To compute $Z_{irr}\cdot \delta_2^{\emptyset}$, we pull back to $\overline{\calM}_{g-2,n+1}\times \overline{\calM}_{2,1}$ as above. The divisor class $\delta_2^{\emptyset}$ pulls back to $-pr_1^* \psi_{r_1}-pr_2^*\psi_{r_2}$, and on the pullback of $Z_{irr}$ both $r_1$ and $r_2$ are fixed points on fixed components, so $Z\cdot \delta_2^{\emptyset}=0$.

To calculate the intersection with $\delta_{irr}$, we view $\Delta_{irr}$ as the result of gluing the last two marked points on $\overline{\calM}_{g-1,n+2}$, which we call $r_1$ and $r_2$. According to \cite{faberalgorithms}, the pullback of $\delta_{irr}$ to itself is equal to
$$
-\psi_{r_1}-\psi_{r_2}+\delta_{irr}+\displaystyle\sum_{h=0}^{g-1}\displaystyle\sum_{1\in Q\subseteq I; (h,Q)\ne (g-1,I)}
 \left( \delta_h^{Q\cup \lbrace r_1\rbrace}+\delta_h^{Q\cup \lbrace r_2\rbrace}\right).
$$
The pullback $Y_{irr}$ of the curve $Z_{irr}$ to $\overline{\calM}_{g-1,n+2}$ consists of the fixed genus $g-2$ curve $(C_1,p_1,\ldots,p_n,q_1)$ containing the first $n$ marked points, an elliptic curve $(E,q_2,r_2)$ containing the last varying marked point $r_2$, and a rational tail connecting $q_1$ and $q_2$, and containing $r_1$. Therefore, the curve $Y_{irr}$ lies in the boundary divisors $\delta_1^{\{r_2\}}=\delta_{g-2}^{I\cup \{r_1\}}$ and $\delta_1^{\{r_1,r_2\}}=\delta_{g-2}^I$, intersects at one point the boundary divisor $\delta_1^{\emptyset}=\delta_{g-2}^{I\cup\{r_1,r_2\}}$ (when $r_2$ hits $q_2$ it moves off of $E$ onto a second rational bridge), and does not intersect the other boundary divisors. Also, $Y_{irr}$ has zero intersection with $\psi_{r_1}$ but may have non-trivial intersection with $\psi_{r_2}$.

Looking at the formula above, we are only interested in $Y_{irr}\cdot \psi_{r_2}$ and $Y_{irr}\cdot \delta_1^{\{r_2\}}$. To compute $Y_{irr}\cdot \delta_{1}^{\{r_2\}}$, we pull back a second time to $\overline{\calM}_{g-2,n+2}\times\overline{\calM}_{1,2}$. Let $t_1$ and $t_2$ denote the points of attachment, then $\delta_1^{\{r_2\}}$ pulls back to $-pr_1^*\psi_{t_1}-pr_2^*\psi_{t_2}$. The curve $Y_{irr}$ splits at the point $q_2$ and pulls back to a fixed marked curve in the first factor and the elliptic curve $(E,r_2,t_2)$, with the point $r_2$ moving along $E$. The divisor $\psi_{t_2}=K_{t_2}+\delta_0^{\{r_2,t_2\}}$ integrates to one on the second factor, therefore
$$
Y\cdot \delta_1^{\{r_2\}}=-1.
$$
Similarly, we see that $Y_{irr}\cdot \psi_{r_2}=0$. Indeed,
$$
\psi_{r_2}=K_{r_2}+\displaystyle\sum_{Q\subset I\cup\{r_1\},Q\neq \emptyset} \delta_0^{Q\cup \{r_2\}},
$$
none of these boundary divisors intersect $Y_{irr}$, and $K_{r_2}$ integrates to zero on $Y_{irr}$ because it is an elliptic curve. Putting all this together, we get that $Z_{irr}\cdot\delta_{irr}=-1$.

Finally, the Hodge bundle on $Z_{irr}$ is trivial, so $Z_{irr}\cdot \lambda_1=0$, and $Z_{irr}\cdot K_i=0$ because the marked points are all fixed on a fixed component.
\end{proof}

A straightforward computation (noting that the matrix of intersections above is close to being diagonal) shows that the matrix of intersections of our test curves with the chosen basis of $\Pic_\QQ(\oMgn)$ is non-degenerate, and we thus get
\begin{cor}\label{cor:complete}
The curve classes $Z_j$, $Z_l^Q$, $\Ell$ and $Z_{irr}$ generate over $\QQ$ the group $N_1(\oMgn)$ of numerical equivalence classes of curves on $\oMgn$.
\end{cor}

\section{The class of the theta divisor}

In this section, we compute the intersection numbers of the test curves defined in the previous section with the pullbacks $[s_\ud^*\Theta]$ and $[s_\ud^*T]$ of the theta divisors, and then prove the main theorems.
\begin{prop} \label{prop:thetaintersect}
For $\deg\ud=0$, we have
$$
Z_i\cdot[s_\ud^*T]=d_i^2 g, \ Z_h^P\cdot[s_\ud^*T]=d_P^2(g-h).
$$
For $\deg\ud=g-1$, we have
$$
Z_i\cdot[s_\ud^*\Theta]=d_i^2 g, \ Z_h^P\cdot[s_\ud^*\Theta]=(d_P-h)^2(g-h)
$$
(we compute the intersections with $\Ell$ and $Z_{irr}$ separately).
\end{prop}
\begin{proof}
We compute these intersection numbers by noting that the restrictions of $s_\ud$ to the test curves can be understood as Abel--Jacobi embeddings. We first compute, for $\deg\ud=0$, the intersections with $[s_\ud^*T]$.

Since $\pi(Z_i)=C$ is a single point in $\calM_g$, the image $s_\ud(Z_i)$ lies inside $\Pic^0(C)$. The restriction of $s_\ud$ to $Z_i$ is therefore the composition of an Abel--Jacobi embedding $Z_i\rightarrow\Pic^0(C)$  and a multiplication by $d_i$ map on $\Pic^0(C)$. The theta function restricted to the Abel--Jacobi image of the curve has degree $g$, while the pullback under the multiplication has degree $d_i^2$ on divisors. Therefore we get $Z_i\cdot[s_\ud^* T]=d_i^2 g$.

Similarly, the Jacobian variety of any curve parameterized by $Z_h^P$ is $\Pic^0(C_1)\times\Pic^0(C_2)$. The limit of the Abel--Jacobi embedding is a more delicate issue. Indeed, note that the Abel-Jacobi mapping is naturally an embedding $C\to\Pic^1(C)$. Thus to have a map $C\to\Pic^0(C)$, we need to choose a basepoint for the embedding. In a family of curves degenerating to some $C_1\cup C_2$ with one node, the limit of the chosen base point must lie on both $C_1$ and $C_2$, for the limit of the Abel--Jacobi embedding to be well-defined. Therefore it must be the point $q_1\in C_1$, which is identified with $q_2\in C_2$ to form the node. For the case of $\deg\ud=0$, we thus have in the limit
\begin{equation}\label{eq:AbelJacobi0}
s_\ud((C_1,p_{j_1},\ldots,p_{j_{n-m}},q_1),(C_2,p_{i_1},\ldots,p_{i_m},q_2))=
\end{equation}
$$
=\left(\displaystyle\sum_{j\in P^c}d_j p_j+d_P q_1,\displaystyle\sum_{i\in P}d_i p_i-d_{P}q_2\right)\in \Pic^0(C_1)\times\Pic^0(C_2),
$$
We recall that the theta function on a decomposable abelian variety is the product of the theta functions on the two factors. The point $q_1$ varies along the curve $C_1$ while $q_2$ is fixed, hence the second term is a constant map. The first term is the composition of the Abel--Jacobi embedding of $C_1$ with a multiplication by $d_P$, and thus has degree $d_P^2(g-h)$.


\smallskip

The computations for the pullback $[s_\ud^*\Theta]$, for the case $\deg\ud=g-1$, are similar.
First, the intersections with $Z_i$ are the same as for $[s_\ud^*T]$, as the degree $0$ and $g-1$ Picard varieties of a curve with rational tails are isomorphic.

The intersection numbers of $Z_h^P$ with $[s^*_{\ud}\Theta]$ are not the same, because the limit of the Abel--Jacobi map in degree $g-1$ is different. As a smooth curve $C$ degenerates to a nodal curve $C_1\cup C_2$, the Jacobian variety $\Pic^{g-1}(C)$ becomes  $\Pic^{g-h-1}(C_1)\times\Pic^{h-1}(C_2)$. Therefore,
\begin{equation}\label{eqn:AbelJacobig-1}
s_\ud((C_1,p_{j_1},\ldots,p_{j_{n-m}},q_1),(C_2,p_{i_1},\ldots,p_{i_m},q_2))=
\end{equation}
$$  =\left(\displaystyle\sum_{j\in P^c}d_j p_j+(d_P-h) q_1,\displaystyle\sum_{i\in P}d_i p_i-(d_{P}-h+1)q_2\right)
$$
$$
\in \Pic^{g-h-1}(C_1)\times\Pic^{h-1}(C_2),
$$
and by the same reasoning we obtain $Z_h^P\cdot[s_\ud^*\Theta]=(d_P-h)^2 (g-h)$.
\end{proof}

We are now ready to prove our main result.

\begin{thm} \label{thm:main}
For $\deg \ud=0$, the class $[s_\ud^*T]\in \Pic_\QQ(\oMgn)$ of the pullback of the universal symmetric theta divisor trivialized along the zero section is equal to
\begin{equation}\label{eq:formulaforT}
[s_\ud^*T]=0\cdot \lambda_1+0\cdot\delta_{irr}+\frac{1}{2}\displaystyle\sum_{i=1}^n d_i^2 K_i-\frac{1}{2}\displaystyle\sum_{P\subseteq I}
\left(d_P^2-\displaystyle\sum_{i\in P}d_i^2\right)\delta_0^P-\frac{1}{2}
\!\!\!\!\!\sum_{\ \ h> 0,P\subseteq I}\!\!\!\!\! d_P^2
\delta_h^P.
\end{equation}
We denote $d_P=\sum_{i\in P}d_i$, and the last sum includes every boundary divisor class exactly once.

For $\deg \ud=g-1$, the class $[s_\ud^*\Theta]\in \Pic_\QQ(\oMgn)$ of the pullback of the universal theta divisor is equal to
\begin{equation}\label{eq:formulaforTheta}
[s_\ud^*\Theta]=-\lambda_1+\frac18\delta_{irr}+\frac{1}{2}\displaystyle\sum_{i=1}^n d_i(d_i+1) K_i-\frac{1}{2}\displaystyle\sum_{P\subseteq I}\left(d_P^2-\displaystyle\sum_{i\in P}d_i^2\right)\delta_0^P-\\
\end{equation}
$$
-\frac{1}{2}\displaystyle\sum_{h> 0,P\subseteq I} (d_P-h)(d_P-h+1)
\delta_h^P.
$$
Moreover, the class of the divisor $\overline{D}_{\ud}$ considered by M\"uller \cite{muller} is expressed in terms of $[s^*_{\ud}\Theta]$ as
\begin{equation}\label{eq:ThetaandD}
 [\overline{D}_{\ud}]=[s_\ud^*\Theta]-\displaystyle\sum_{P\subset P_+,d_P<h}(h-d_P)\delta_h^P-\delta_{irr}/8,
\end{equation}
where $P_+:=\{i\in I \mid d_i>0\}$, thus reproving the formula in Theorem~\ref{thm:muller}.
\end{thm}
\begin{proof}
According to Corollary \ref{cor:complete}, the curve classes $Z_j$, $Z_h^P$, $\Ell$, and $Z_{irr}$ generate $N_1(\oMgn)$. Moreover, the test curves $Z_j$ and $Z_h^P$ have zero intersection with the divisors $\lambda_1$ and $\delta_{irr}$. Therefore to calculate the coefficients of the divisors $K_i$ and $\delta_h^P$ in the theta divisors it suffices to intersect both sides of (\ref{eq:formulaforT}) and (\ref{eq:formulaforTheta}) with the test curves $Z_j$ and $Z_h^P$, using Propositions \ref{prop:standardintersect} and \ref{prop:thetaintersect}, and to verify that they are equal. This is a tedious but straightforward calculation.

Computing the remaining coefficients is somewhat more complicated, since the test curves $\Ell$ and $Z_{irr}$ parametrize curves of non-compact type, and we need to understand the degeneration of the Abel--Jacobi map to such curves.

We first compute the remaining coefficients in $[s^*_{\ud}T]$. To calculate the intersection number $\Ell\cdot [s^*_{\ud}T]$, note that a generic point of $\Ell$ is a smooth elliptic curve $C_1$ attached to a fixed genus $g-1$ curve $C_2$ with all the marked points on it. According to (\ref{eq:AbelJacobi0}), the limit of the Abel--Jacobi map at that point is $(0,\sum d_i p_i)\in \op{Pic}^0(C_1)\times \op{Pic}^0(C_2)$, where the second term is constant. The value of the theta function trivialized along the zero section at all such points is the constant number $\theta(\tau,\sum d_i p_i)/\theta(\tau,0)$, where $\tau$ is the period matrix of $C_2$, and we can choose the marked points $p_i$ so that it is non-zero. This number does not change as the elliptic curve degenerates, therefore, $\Ell\cdot [s^*_{\ud}T]=0$.

The calculation of $Z_{irr}\cdot [s^*_{\ud}T]$ is similar. Consider a reducible curve consisting of a smooth genus $g-2$ component $(C_1,p_1,\ldots,p_n,q_1)$ attached to a smooth genus two component $(C_2,q_2)$. The value of the theta function trivialized along the zero section at such a curve depends only on the position of the marked points on the first component. Any point in $Z_{irr}$ lies in a one-parameter family consisting of such curves, hence the theta function is constant on $Z_{irr}$ and $Z_{irr}\cdot [s^*_{\ud}T]=0$. Intersecting both sides of (\ref{eq:formulaforT}) with $\Ell$ and $Z_{irr}$ and using Proposition \ref{prop:standardintersect} shows that the coefficient of $\delta_{irr}$  in $[s^*_{\ud}T]$ is also zero.

Before finding the remaining coefficients of $[s^*_{\ud}\Theta]$, we first prove formula (\ref{eq:ThetaandD}). Over the smooth locus $\calM_{g,n}$ the classes $D_\ud$ and $s^*_\ud\Theta$ coincide by definition. However, the theta function vanishes identically on certain boundary components, so the divisor $s_\ud^*\Theta$ is reducible and contains some boundary components with multiplicities that we now compute.

Consider a boundary divisor $\Delta_h^P=\overline{\calM}_{g-h,n-m+1}\times\overline{\calM}_{h,m+1}$. The restriction of the theta function to a generic point of this divisor is given by formula (\ref{eqn:AbelJacobig-1}). It may happen that on one of the two components all marked points have non-negative weights. In this case the image of the divisor on that component lies entirely inside the theta divisor of the corresponding Jacobian. Assume without loss of generality that this happens on the second component. The order of vanishing of the theta function on such a boundary divisor is equal to $h-d_P$ by the Riemann theta singularity theorem, therefore to relate $[\overline{D}_{\ud}]$ and $[s^*_{\ud}\Theta]$ we need to subtract the corresponding multiple of $\delta_h^P$.

Along $\Delta_{irr}$, the generic vanishing order of the theta function is equal to $1/8$. This can be seen by looking at the Fourier--Jacobi expansion of the theta function, see~\cite{donagiscju} for the relevant computation (note that we are dealing with the actual universal theta function and not with the polarization on semiabelic varieties). Subtracting all the generic vanishing we get
$$
 [\overline{D}_{\ud}]=[s_\ud^*\Theta]-\displaystyle\sum_{P\subset P_+,d_P<h}(h-d_P)\delta_h^P-\delta_{irr}/8,
$$
proving formula (\ref{eq:ThetaandD}).

It remains to compute the coefficients of $\lambda_1$ and $\delta_{irr}$ in $[s^*_{\ud}\Theta]$. First, we intersect with the curve $\Ell$. Once all the generic vanishing is taken out, the restriction of $\overline {D}_\ud$ and of $s_\ud^*T$ to the curve $\Ell$ are the same (since on an elliptic curve $g-1=0$), and thus we have $[\overline{D}_{\ud}]\cdot \Ell=0$ (compare also to \cite[Lemma 4.2]{muller}). Denoting the coefficients of $\lambda_1$ and $\delta_{irr}$ in $[s^*_{\ud}\Theta]$ by $a$ and $b$, respectively, and using (\ref{eq:ThetaandD}) and the intersection numbers from Proposition \ref{prop:standardintersect}, we see that $a+12b=1/2$.

To finish the computation of the class $[s^*_{\ud}\Theta]$, we thus need to find one more relation. Intersecting with the test curve $Z_{irr}$ is tricky, as it requires having an explicit description of the behavior of the map $s_\ud$ on irreducible stable curves. This can be accomplished by a careful study of the Fourier-Jacobi expansion, but we take another approach. Consider the restriction of the class $[s^*_{\ud}\Theta]$ to the image of the gluing map $i:\calM_{2,1}^{ct}\times\overline{\calM}_{g-2,n+1}\to\overline{\calM}_{g,n}$. We note that on $\calM_2^{ct}$ (and thus also by pullback on $\calM_{2,1}^{ct}$) the class $\lambda_1$ is a boundary, namely $\lambda_1=\delta_1^{\emptyset}/5$ (see \cite{hamobook}, p. 171). Therefore, computing the pullback $i^*[s^*_{\ud}\Theta]$ gives us the one extra condition that we need. Using the pullback formulas described in \cite{faberalgorithms}, we see that
$$
i^*[s^*_{\ud}\Theta]=pr_1^*(a\lambda_1+\psi)+pr_2^*(\cdots),
$$
where the second term is not relevant for us.

To compute this class geometrically, we consider the pullback of the universal theta divisor $\theta\subset\PPic^1$ from the universal Picard variety over $\calM^{ct}_2=\calA_2$  (where $\theta$ is the locus of effective divisors) to $\calM^{ct}_{2,1}$, under the map $s_1(C,p):=p\in\Pic^1(C)$. The image of $s_1$ is exactly $\theta$, so we compute its pullback using the adjunction formula. Since we are dealing with stacks, we need to take automorphisms into account. The generic points of $\calM_{2,1}^{ct}$ and of $\PPic^1$ have trivial stabilizers. The moduli space $\calM_{2,1}^{ct}$ has two divisors whose generic points have non-trivial stabilizers, the locus of Weierstrass points and the boundary divisor $\delta_1^{\emptyset}$. On $\PPic^1$, the images of these two divisors under $s_1^*$ are the locus of $2$-torsion points and the subfamily of the universal family on $\calA_{1,1}$, the moduli space of products of elliptic curves, that is trivial on one of the factors. The generic stabilizer group both on the source and the target divisor is $\ZZ/2\ZZ$ in each of these two cases, so the map $s^*_{\ud}$ is unramified in codimension one, and therefore
$$
  K_{\calM^{ct}_{2,1}}=(K_{\PPic^1}+\theta)|_\theta.
$$
The canonical class of $\calM^{ct}_{2,1}$ is equal to $13\lambda_1-2\delta_1^{\emptyset}+\psi$ (see \cite[Thm.~7.15]{acgh2}), therefore the canonical class of $\calM^{ct}_{2,1}$ is $K_{\calM^{ct}}$ plus the first Chern class of the cotangent bundle of the fiber, which is $\psi$. 
The cotangent bundle to $\PPic^1$ is the extension of the pullback of the cotangent bundle of ${\calA_2}$ by the cotangent bundle of an individual $\Pic^1$; globally, this means that the canonical bundle on $\PPic^1$ is the sum of the pullback of $K_{\calA_2}$ (which is $3\lambda$) and the first Chern class of the Hodge bundle. Thus we have  $K_{\PPic^1}=4\lambda_1$, and finally obtain
$$
 \theta|_\theta=K_{\calM^{ct}_{2,1}}-K_{\PPic^1}|_\theta=\psi+13\lambda_1-2\delta_1^\emptyset-4\lambda_1=\psi+9\lambda_1-2\delta_1^\emptyset
 =-\lambda_1+\psi,
$$
where we use $5\lambda_1=\delta_1^\emptyset$ on $\calM_{2,1}$ for the last step. It follows that the coefficient $a$ of the class $[s^*_\ud\Theta]$ is equal to $-1$, completing the proof.
\end{proof}

\begin{proof}[Proof of Theorem~\ref{thm:0}]
The theorem directly follows from the formula for $[s^*_\ud T]$, restricted to the moduli space $\calM_{g,n}^{ct}$ of curves of compact type over which the universal family $\calX_g\rightarrow\calA_g$ extends, and from formula (\ref{equ:0section}).
\end{proof}

\medskip {\em Acknowledgment:}
The second author would like to thank Maksym Fedorchuk for a number of useful discussions.

\end{document}